\documentclass[12pt,a4paper,reqno]{amsart}

\usepackage[utf8]{inputenc}
\usepackage{mathptmx} 
\DeclareMathAlphabet{\mathcal}{OMS}{cmsy}{m}{n}
\usepackage{amsmath,amssymb,amsthm,amsfonts}
\usepackage{latexsym}
\usepackage{mathrsfs}
\usepackage{bbm}
\usepackage{titlesec}
\usepackage{xcolor}
\usepackage{url}
\usepackage{graphicx}
\usepackage{stmaryrd}
\usepackage[pdfencoding=auto]{hyperref}
\usepackage{graphicx}

\newtheorem{thm}{Theorem}[section]
\newtheorem{cor}[thm]{Corollary}
\newtheorem{lem}[thm]{Lemma}
\newtheorem{prop}[thm]{Proposition}

\theoremstyle{definition}

\newtheorem{rem}[thm]{Remark}

\numberwithin{equation}{section}

\titleformat{\section}{\normalfont\bfseries\centering}{\thesection.}{.25em}{}
\titleformat{\subsection}{\normalfont\bfseries}{\thesubsection.}{.25em}{}
\titleformat{\subsubsection}{\normalfont\it}{\thesubsubsection.}{.25em}{}
\titlespacing{\section}{0pt}{*4}{*1.5}
\titlespacing{\subsection}{0pt}{*4}{*0.5}
\allowdisplaybreaks

\newcommand{\braces}[1]{{\rm (}#1{\rm )}}

\newcommand{\<}{\langle}
\renewcommand{\>}{\rangle}

\newcommand{\wh}{\widehat}

\newcommand{\R}{\ensuremath{\mathbb R}}    
\newcommand{\N}{\ensuremath{\mathbb N}}    

\newcommand{\linspan}{\operatorname{span}}

\newcommand{\bP}{\mathbb P}
\newcommand{\bE}{\mathbb E}

\newcommand{\diag}{\operatorname{diag}}

\newcommand{\Tr}{\operatorname{Tr}}
\newcommand{\bH}{\mathbb H}
\newcommand{\cinc}{\!\lhook\joinrel\xrightarrow{\hspace{.1cm}c\hspace{.1cm}}}

\begin{document}
\title[Koopman-invariance of Gaussian RKHSs]{Invariance of Gaussian RKHSs under Koopman operators of stochastic differential equations with constant matrix coefficients}

\author[F.~Philipp]{Friedrich Philipp}
\address{{\bf F.~Philipp}
	Technische Universit\"at Ilmenau, Institute of Mathematics,
	Weimarer Stra\ss e 25, D-98693 Ilmenau, Germany}
\email{\href{mailto:friedrich.philipp@tu-ilmenau.de}{friedrich.philipp@tu-ilmenau.de}}
\urladdr{\url{www.fmphilipp.de}}

\author[M.~Schaller]{Manuel Schaller}
\address{{\bf M.~Schaller}
	Technische Universit\"at Ilmenau, Institute of Mathematics,
	Weimarer Stra\ss e 25, D-98693 Ilmenau, Germany}
\email{\href{mailto:manuel.schaller@tu-ilmenau.de}{manuel.schaller@tu-ilmenau.de}}
\urladdr{\url{www.tu-ilmenau.de/deq/manuel-schaller}}

\author[K.~Worthmann]{Karl Worthmann}
\address{{\bf K.~Worthmann}
	Technische Universit\"at Ilmenau, Institute of Mathematics,
	Weimarer Stra\ss e 25, D-98693 Ilmenau, Germany}
\email{\href{mailto:karl.worthmann@tu-ilmenau.de}{karl.worthmann@tu-ilmenau.de}}
\urladdr{\href{https://www.tu-ilmenau.de/universitaet/fakultaeten/fakultaet-mathematik-und-naturwissenschaften/profil/institute-und-fachgebiete/institut-fuer-mathematik/profil/fachgebiet-optimization-based-control/team/karl-worthmann}{Karl Worthmann}}

\author[S.~Peitz]{Sebastian Peitz}
\address{{\bf S.~Peitz}
	Paderborn University, Department of Computer Science, Data Science for Engineering, Germany}
\email{\href{mailto:sebastian.peitz@upb.de}{sebastian.peitz@upb.de}}
\urladdr{\url{www.uni-paderborn.de/en/person/47427}}

\author[F.~N\"uske]{Feliks N\"uske}
\address{{\bf F.~N\"uske}
	Max Planck Institute for Dynamics of Complex Technical Systems, Magdeburg, Germany, and Freie Universität Berlin, Berlin, Germany}
\email{\href{mailto:nueske@mpi-magdeburg.mpg.de}{nueske@mpi-magdeburg.mpg.de}}
\urladdr{\url{www.mi.fu-berlin.de/en/math/groups/ai4s/staff/nueske.html}}

\begin{abstract}
We consider the Koopman operator semigroup $(K^t)_{t\ge 0}$ associated with stochastic differential equations of the form $dX_t = AX_t\,dt + B\,dW_t$ with constant matrices $A$ and $B$ and Brownian motion $W_t$. We prove that the reproducing kernel Hilbert space $\bH_C$ generated by a Gaussian kernel with a positive definite covariance matrix $C$ is invariant under each Koopman operator $K^t$ if the matrices $A$, $B$, and $C$ satisfy the following Lyapunov-like matrix inequality: $AC^2 + C^2A^\top\le 2BB^\top$.  In this course, we prove a characterization concerning the inclusion $\bH_{C_1}\subset\bH_{C_2}$ of Gaussian RKHSs for two positive definite matrices $C_1$ and $C_2$. The question of whether the sufficient Lyapunov-condition is also necessary is left as an open problem.
\end{abstract}

\maketitle

\section{Introduction}
The Koopman operator~\cite{Koop31} of a (stochastic) dynamical system is a linear operator which is defined on a linear function space of so-called {\em observables}. For deterministic systems, it is simply defined as a composition operator with the flow $F$ of the system, i.e., $(K^tf)(x) = f(F(x,t))$. For stochastic systems, $(K^tf)(x)$ is defined as the conditional expectation of $f(F(\cdot,t))$, given that the trajectory starts at $x$. Since the (generally unknown) Koopman operator is linear and provides complete information about the (expected) process behavior, it is very appealing to approximate this operator using sample data in the fashion of modern machine learning, see e.g. the monograph \cite{KoopmanBook} and the references therein.

One of the most popular methods for Koopman operator learning is certainly Extended Dynamic Mode Decomposition (EDMD) \cite{WillKevr15}, a data-driven approach propagating a finite number of pre-defined observable functions along the flow, which results in a well-interpretable surrogate model for analysis, prediction, and control. EDMD has been successfully applied in a number of highly relevant applications such as molecular dynamics~\cite{SchuetteKoltaiKlus2016}, 
turbulent flows \cite{GianKolc18}, 
neuroscience~\cite{BrunJohn16}, or climate prediction~\cite{AzenEri20}, just to name a few. Rigorous error analyses for EDMD have been conducted in \cite{Mezi22,ZhanZuaz23,NuskPeit23,PhilScha24b}.

Since kernel methods play an important role in approximation theory and machine learning, it is not surprising that there exist kernel-based approaches for Koopman operator learning, one of them being a variant of EDMD, called kernel EDMD (kEDMD)~\cite{WRK15kernel,KSM20}. Kernel-based methods have been rigorously analyzed in, e.g., \cite{PhilScha24a,PhilScha23,KostLoun23,KostNov22} and in \cite{KoehPhil24} for deterministic systems, where uniform error bounds have been provided. As it turns out in the analysis, for providing error bounds on Koopman approximations in the kernel-based setting it is highly beneficial if the reproducing kernel Hilbert space (RKHS) generated by the kernel at hand is invariant under the Koopman operator. Therefore, an important task in kernel-based Koopman operator learning is to figure out which kernels and which systems are a good match in the sense that the generated RKHS is invariant under the Koopman operator of the system.

In this article, we consider the invariance of Gaussian reproducing kernel Hilbert spaces under the Koopman operator of dynamical systems driven by a stochastic differential equation (SDE) of the form $dX_t = AX_t\,dt + BdW_t$ with constant coefficient matrices $A$ and $B$, where $(A,B)$ is controllable and $A$ is Hurwitz. To be more precise, we prove that the RKHS $\bH_C$ generated by the Gaussian kernel $k(x,y) = \exp(-\|C^{-1}(x-y)\|^2)$ with a positive definite matrix $C$ is Koopman-invariant if the matrices $A$, $B$, and $C$ satisfy the  Lyapunov-like inequality $AC^2 + C^2A^\top\le 2BB^\top$, see Theorem \ref{t:invariance}. It is left open whether this sufficient condition is also necessary. In the course of proving Theorem \ref{t:invariance}, we also prove that $\bH_{C_1}\subset\bH_{C_2}$ holds for two positive definite matrices $C_1$ and $C_2$ if and only if $C_1^2\ge C_2^2$, see Proposition \ref{p:inc_gen}. This fact is well known \cite{shs} for the scalar case, where $C_1,C_2\in (0,\infty)$, but---to the best of the authors' knowledge---seems to be unknown in the general matrix case.

The paper is arranged as follows. In the next section, we briefly introduce the reader to the Koopman operator of SDEs and present the main result, Theorem \ref{t:invariance}. In Section \ref{s:gaussian}, we consider Gaussian kernels and their corresponding RKHSs and make use of the auxiliary results therein in Section \ref{s:proof} which is dedicated to the proof of Theorem \ref{t:invariance}.

\section{Setting and main result}
Recall that a pair of matrices $(A,B)$ with $A\in\R^{d\times d}$ and $B\in\R^{d\times m}$, $m\le d$, is called {\em controllable} if $(B,AB,\ldots,A^{d-1}B)$ has full rank $d$. The notion is due to the fact that a controlled system $\dot x(t) = Ax(t) + Bu(t)$ can be steered from any state $x(0)=x_0$ in any time $t>0$ to any state $z\in\R^d$ by a control function $u$ if and only if $(A,B)$ is controllable. Also recall that a square matrix $A$ is called {\em Hurwitz} if the real parts of the eigenvalues of $A$ are all negative. This implies that $\|e^{At}\|\le Me^{-\omega t}$ for all $t\ge 0$ with some constants $M,\omega>0$.

Let us consider a $d$-dimensional stochastic differential equation with constant matrix coefficients
\begin{align}\label{e:sde}
dX_t = AX_t\,dt + B\,dW_t,    
\end{align}
where $W_t$ is $m$-dimensional Brownian motion, $m\le d$, and $A\in\R^{d\times d}$, $B\in\R^{d\times m}$. Here, we assume that $(A,B)$ is controllable and that $A$ is Hurwitz. The solution process $X_t$ emerging from the SDE \eqref{e:sde} is also called a $d$-dimensional {\em Ornstein-Uhlenbeck process}. It has the unique stationary distribution
\[
d\mu(x) = (2\pi)^{-d/2}(\det\Sigma)^{-1/2}\exp(-\tfrac 12x^\top\Sigma^{-1}x)\,dx,
\]
where $\Sigma\in\R^{d\times d}$ is the positive definite matrix
\[
\Sigma = \int_0^\infty e^{As}BB^\top e^{A^\top s}\,ds.
\]
Note that the limit $\Sigma$ exists since $A$ is Hurwitz and that $\Sigma$ is positive definite thanks to controllability of $(A,B)$. In linear control theory, the matrix $\Sigma$ is called the {\em controllability Gramian} and is a solution of the matrix equation $AX+XA^\top = -BB^\top$.

The solution process $X_t$ is a time-homogeneous Markov process. Recall that every such process has a so-called Markov transition kernel $\rho_t(x,F) = \bP(X_t\in F\,|\,X_0=x)$, where $x\in\R^d$ and $F$ is a Borel set. Here, the transition kernel is absolutely continuous w.r.t.\ Lebesgue measure, and its density is given by (see \cite{vp})
\begin{align*}
\rho_t(x,dy)
&= (2\pi)^{-d/2}(\det\Sigma(t))^{-1/2}\exp\big(-\tfrac 12 (y-e^{At}x)^\top\Sigma(t)^{-1}(y-e^{At}x)\big)\,dy,
\end{align*}
where
\[
\Sigma(t) = \int_0^te^{As}BB^\top e^{A^\top s}\,ds.
\]
Again, $\Sigma(t)$ is positive definite as $(A,B)$ is controllable.

The {\em Koopman operator} $K^t : L^2(\mu)\to L^2(\mu)$ at time $t\ge 0$ corresponding to the SDE \eqref{e:sde} is defined by
\begin{align}\label{e:koopman}
(K^tf)(x) := \bE[f(X_t)\,|\,X_0=x] = \int_{\R^d}f(y)\,\rho_t(x,dy),\qquad f\in L^2(\mu).
\end{align}
It is well known (see, e.g., \cite[Proposition 2.8]{PhilScha24a}) that $(K^t)_{t\ge 0}$ forms a strongly continuous semigroup of contractions on $L^2(\mu)$.

Here, we consider Gaussian kernels on $\R^d$ with a positive definite covariance matrix $C\in\R^{d\times d}$, i.e.,
\begin{equation*}
k^C(x,y) = \exp\big[-(x-y)^\top C^{-2}(x-y)\big] = \exp\big[-\|C^{-1}(x-y)\|^2\big].
\end{equation*}
Note that for $C = \sigma\cdot I_d$, $\sigma>0$, we obtain the usual RBF kernel $k^\sigma(x,y) = \exp(-\|x-y\|^2/\sigma^2)$. By $\bH_C$ we denote the RKHS generated by the kernel $k^C$. The Hilbert space norm on $\bH_C$ will be denoted by $\|\cdot\|_C$. For $\sigma>0$ we simply write $\bH_\sigma$ and $\|\cdot\|_\sigma$ instead of $\bH_{\sigma I}$ and $\|\cdot\|_{\sigma I}$, respectively.

The main result of this note reads as follows.

\begin{thm}\label{t:invariance}
Let $C\in\R^{d\times d}$ be a symmetric positive definite matrix such that
\begin{align}\label{e:cond}
\tfrac 12\big(AC^2+C^2A^\top\big)\,\le\,BB^\top.
\end{align}
Then the RKHS $\bH_C$ is invariant under each Koopman operator $K^t$, $t\ge 0$, associated with the SDE \eqref{e:sde}, and we have
\[
\|K^tf\|_C\le e^{t\cdot\Tr(-A)/2}\cdot\|f\|_C,\qquad f\in\bH_C.
\]
\end{thm}

Let us briefly discuss the Lyapunov-like condition \eqref{e:cond} on $C$.

\begin{rem}
(a) Since $A$ is assumed to be Hurwitz, for every positive definite matrix $Q\in\R^{d\times d}$ there exists a positive definite matrix $P\in\R^{d\times d}$ such that $AP+PA^\top = -Q$. Hence, the set of positive definite matrices $C$ satisfying the matrix inequality \eqref{e:cond} is not empty. For example, $C = \alpha\Sigma^{1/2}$ is a solution for every $\alpha>0$.

\smallskip\noindent
(b) If $m=d$ and $B$ is invertible, then for any given positive definite matrix $C$ there exists $\tau>0$ such that $\tau C$ satisfies \eqref{e:cond}. Indeed, if this was not the case, there would exist a sequence $(x_n)\subset\R^d$ such that $\frac 1n\<(AC^2 + C^2A^\top)x_n,x_n\> > \|B^\top x_n\|^2$. It is no restriction to assume $\|x_n\|=1$ and $x_n\to x$ as $n\to\infty$ for some $x\in\R^d$, $\|x\|=1$. But then, the above implies $B^\top x=0$ and hence $x=0$, a contradiction.

\smallskip\noindent
(c) It is not clear whether the condition \eqref{e:cond} is also necessary for Koopman invariance of $\bH_C$. We discuss this at the end of Section \ref{s:proof}.
\end{rem}

\begin{cor}
Let $\sigma>0$ and assume that $\frac 12(A+A^\top)\le \frac 1{\sigma^2}BB^\top$ \braces{e.g., if $A$ is dissipative\footnote{Recall that a matrix $A\in\R^{d\times d}$ is called {\em dissipative} if $A+A^\top\le 0$.}}. Then the RKHS $\bH_\sigma$ is invariant under each Koopman operator $K^t$, $t\ge 0$, associated with the SDE \eqref{e:sde}, and we have
\[
\|K^tf\|_\sigma\le e^{t\cdot\Tr(-A)/2}\|f\|_\sigma,\qquad f\in\bH_\sigma.
\]
\end{cor}

We conclude this section by tying up some notations. First of all, we agree on the convention that a positive definite kernel is always symmetric, i.e., $k(x,y) = k(y,x)$, and real-valued. Also, the term ``positive definite'' only refers to real, symmetric matrices. Let $X$ be a set. For $x\in X$ and a positive definite kernel $k$ on $X$, set $k_x(y) := k(x,y)$, $y\in X$. That is, $x\mapsto k_x$ is the canonical feature map of the kernel $k$. For two positive definite kernels $k_1$ and $k_2$ on $X$ we write $k_1\preceq k_2$ if
$$
\sum_{i,j=1}^n\alpha_i\alpha_j k_1(x_i,x_j)\le \sum_{i,j=1}^n\alpha_i\alpha_j k_2(x_i,x_j)
$$
for any choice of $n\in\N$, $\alpha_j\in\R$, and $x_j\in X$, $j=1,\ldots,n$. We also write $V\cinc W$ (continuous embedding) for two normed vector spaces $V$ and $W$ to indicate that $V\subset W$ and the identity map from $V$ to $W$ is continuous, i.e., if there exists $K>0$ such that $\|v\|_W\le K\|v\|_V$ for all $v\in V$.

\section{Some properties of Gaussian kernels}\label{s:gaussian}
In this section, we collect the statements on Gaussian kernels that are utilized in the proof of Theorem \ref{t:invariance}. The next lemma shows that $\bH_C$ can be regarded as a dense subspace of $L^2(\mu)$ and that the $\bH_C$-norm is stronger than the $L^2(\mu)$-norm on $\bH_C$.

\begin{lem}
Let $C\in\R^{d\times d}$ be positive definite. Then the RKHS $\bH_C$ is densely and continuously embedded in $L^2(\mu)$.
\end{lem}
\begin{proof}
The fact that $\bH_C\cinc L^2(\mu)$ is the first part of \cite[Theorem 4.26]{SteiChri08}. Moreover, if $f\in\bH_C$ and $f=0$ $\mu$-a.e., then $f(x)=0$ for all $x\in\R^d$ as $f$ is continuous and the density $d\mu / dx$ is positive. Therefore, we may regard $\bH_C$ as a subspace of $L^2(\mu)$. Concerning the density of $\bH_C$ in $L^2(\mu)$, we have to show that the integral operator $K : L^2(\mu)\to\bH_C$, defined by
\[
Kf(x) = \int_{\R^d}k^C(x,y)f(y)\,d\mu(y),\qquad f\in L^2(\mu),
\]
has trivial kernel in $L^2(\mu)$, see \cite[Theorem 4.26]{SteiChri08}. So, let $f\in L^2(\mu)$ such that $Kf=0$, i.e.,
\[
(\phi * \psi)(x) = \int_{\R^d}\underbrace{\exp(-\|C^{-1}(x-y)\|^2)}_{=\phi(x-y)}\cdot \underbrace{f(y)\exp(-\|Py\|^2)}_{=\psi(y)}\,dy = 0,\qquad x\in\R^d,
\]
where $P = (2\Sigma)^{-1/2}$. Applying the Fourier transform to $\phi * \psi$, we obtain $\wh\phi(\omega)\cdot\wh\psi(\omega)=0$ for a.e.\ $\omega\in\R^d$. But $\wh\phi(\omega) > 0$ for all $\omega\in\R^d$, hence $\wh\psi=0$ and thus $\psi=0$.
\end{proof}

The following lemma is well-known. We will apply it in the proof of Proposition \ref{p:inc_gen} below, which can be seen as a generalization of Lemma \ref{l:ist_bekannt} to the multi-dimensional case.

\begin{lem}\label{l:ist_bekannt}
For $\sigma>0$, consider the Gaussian kernel $k^\sigma$ on $\R$. If $\sigma_1\ge\sigma_2 > 0$, then $\bH_{\sigma_1}\cinc\bH_{\sigma_2}$ with $\|f\|_{\sigma_2}\le(\frac{\sigma_1}{\sigma_2})^{1/2}\|f\|_{\sigma_1}$ for $f\in\bH_{\sigma_1}$, and $k^{\sigma_1}\preceq\frac{\sigma_1}{\sigma_2} k^{\sigma_2}$.
\end{lem}
\begin{proof}
The first part is \cite[Corollary 6]{shs}, and the second follows from Aronszajn's inclusion theorem, Theorem \ref{t:aronszajn}.
\end{proof}

\begin{prop}\label{p:inc_gen}
Let $C_1,C_2\in\R^{d\times d}$ be positive definite. Then $\bH_{C_1}\cinc\bH_{C_2}$ if and only if $C_1^2\ge C_2^2$. In this case, we have
\[
\|f\|_{C_2}\le(\tfrac{\det C_1}{\det C_2})^{1/2}\|f\|_{C_1},\qquad f\in\bH_{C_1},
\]
and $k^{C_1}\preceq\frac{\det C_1}{\det C_2}\cdot k^{C_2}$.
\end{prop}
\begin{proof}
The proof of Proposition \ref{p:inc_gen} is divided into three parts.

{\bf 1.} Denote by $I = I_d$ the $d\times d$-identity matrix. We first prove that for a positive diagonal matrix $D\in\R^{d\times d}$ we have $\bH_{D}\cinc\bH_{I}$ if and only if $D\ge I$. For this, let $D = \diag(\sigma_1,\ldots,\sigma_d)$ with $\sigma_i>0$ for all $i=1,\ldots,d$. Observe that for $x,y\in\R^d$ we have (for the tensor product of kernels see \eqref{e:tensor})
\begin{align*}
k^D(x,y)
&= \exp\big(-(x-y)^\top D^{-2}(x-y)\big) = \exp\Big(-\sum_{k=1}^d\sigma_k^{-2}(x_k-y_k)^2\Big)= \prod_{k=1}^d\exp\Big[-\frac{(x_k-y_k)^2}{\sigma_k^2}\Big]\\
&= \prod_{k=1}^d k^{\sigma_k}(x_k,y_k) = (k^{\sigma_1}\otimes\cdots\otimes k^{\sigma_d})(x,y),
\end{align*}
where the $x_k$ and $y_k$ are the Cartesian coordinates of $x$ and $y$, respectively.

Hence, if $D\ge I$ (i.e., $\sigma_i\ge 1$ for all $i=1,\ldots,d$), Lemmas \ref{l:tensor} and \ref{l:ist_bekannt} imply
\[
k^D = k^{\sigma_{1}}\otimes\cdots\otimes k^{\sigma_{d}}\preceq \sigma_1k^{1}\otimes\cdots\otimes\sigma_dk^1 = \sigma_1\dots\sigma_d\cdot k^I = \det D\cdot k^I,
\]
and Aronszajn's inclusion theorem, Theorem \ref{t:aronszajn}, implies $\bH_D\cinc\bH_I$.

Conversely, let $\bH_D\cinc\bH_I$ and suppose that $\sigma_1 < 1$. By Theorem \ref{t:aronszajn}, we have $k^D\preceq\alpha k^I$ with some $\alpha>0$. By setting $x_k=y_k=0$ for $k=2,\ldots,d$, we see that $k^D(x,y) = k^{\sigma_1}(x_1,y_1)$ and $k^I(x,y) = k^1(x_1,y_1)$. Hence, we have $k^{\sigma_1}\preceq\alpha k^1$. On the other hand, $1 > \sigma_1$ and Lemma \ref{l:ist_bekannt} imply $k^1\preceq\sigma_1^{-1} k^{\sigma_1}$. This implies $\bH_{\sigma_1 I} = \bH_I$, which contradicts \cite[Corollary 7 ii)]{shs}.

{\bf 2.} Next, we prove that for a positive definite matrix $C\in\R^{d\times d}$ we have $\bH_{C}\cinc\bH_{I}$ if and only if $C\ge I$. For this, let $U\in\R^{d\times d}$ be an orthogonal matrix such that $C = U^\top DU$ with a diagonal matrix $D\in\R^{d\times d}$. Then $k^C = k^D\circ U$ and $k^I = k^I\circ U$. If $C\ge I$, then 
\[
k^C = k^D\circ U\preceq\det D\cdot k^I\circ U = \det C\cdot k^I.
\]
Conversely, if $\bH_C\cinc\bH_I$, then $k^C\preceq\alpha k^I$ with some $\alpha>0$, which implies $k^D\preceq\alpha k^I$ and hence $D\ge I$, which is equivalent to $C\ge I$.

{\bf 3.} Finally, let $C_1,C_2\in\R^{d\times d}$ be arbitrary symmetric positive definite matrices. Note that $C_1^2\ge C_2^2$ is equivalent to $C := 
 (C_2^{-1}C_1^2C_2^{-1})^{1/2}\ge I$. Furthermore, note that
\begin{align}\label{e:umrechnung}
k^C(x,y) = k^{C_1}(C_2x,C_2y)
\qquad\text{and}\qquad
k^I(x,y) = k^{C_2}(C_2x,C_2y),\qquad x,y\in\R^d.
\end{align}
Therefore, if $C_1^2\ge C_2^2$, by the second part of the proof,
\[
k^{C_1} = k^C\circ C_2^{-1}\preceq\det C\cdot (k^1\circ C_2^{-1}) = \tfrac{\det C_1}{\det C_2}\cdot k^{C_2},
\]
which yields $\bH_{C_1}\cinc\bH_{C_2}$. Conversely, if $\bH_{C_1}\cinc\bH_{C_2}$, then there exists $\alpha>0$ such that $k^{C_1}\preceq\alpha k^{C_2}$. In view of \eqref{e:umrechnung}, this gives $k^C\circ C_2^{-1} = k^{C_1}\preceq \alpha k^{C_2} = \alpha k^I\circ C_2^{-1}$ and hence $k^C\preceq\alpha k^I$. By part 2, this implies $C\ge I$.
\end{proof}

\begin{rem}
Note that $C_1\ge C_2$ does in general not imply $C_1^2\ge C_2^2$ for symmetric positive definite matrices $C_1$ and $C_2$. On the other hand $C_1\ge C_2$ implies $C_1^{1/2}\ge C_2^{1/2}$ and $C_1^{-1}\le C_2^{-1}$.
\end{rem}

We conclude this section with the following lemma.

\begin{lem}\label{l:rechnen}
Let $C_1,C_2\in\R^{d\times d}$ be positive definite and $z,w\in\R^d$. Then
\[
\int_{\R^d}k_z^{C_1}(y)\cdot k_w^{C_2}(y)\,dy = \frac{\pi^{d/2}}{\det(C_1^{-2}+C_2^{-2})^{1/2}}\cdot k^C(z,w),
\]
where $C = (C_1^2+C_2^2)^{1/2}$.
\end{lem}
\begin{proof}
Let $M := (C_1^{-2}+C_2^{-2})^{1/2}$. We have $k_z^{C_1}(y)k_w^{C_2}(y) = e^{-L(y)}$,
where
\begin{align*}
L(y)
&= (y-z)^\top C_1^{-2}(y-z) + (y-w)^\top C_2^{-2}(y-w)\\
&= y^\top(C_1^{-2}+C_2^{-2})y - 2\big(C_1^{-2}z + C_2^{-2}w\big)^\top y + z^\top C_1^{-2}z + w^\top C_2^{-2}w\\
&= \|My\|^2 - 2\big\<M^{-1}\big(C_1^{-2}z + C_2^{-2}w\big),My\big\> + \|C_1^{-1}z\|^2 + \|C_2^{-1}w\|^2\\
&= \big\|My - M^{-1}(C_1^{-2}z+C_2^{-2}w)\big\|^2 - K(z,w),
\end{align*}
where
\begin{align*}
K(z,w)
&= \big\|M^{-1}(C_1^{-2}z+C_2^{-2}w)\big\|^2 - \|C_1^{-1}z\|^2 - \|C_2^{-1}w\|^2\\
&= \big\<(C_1^{-2}+C_2^{-2})^{-1}(C_1^{-2}z+C_2^{-2}w),(C_1^{-2}z+C_2^{-2}w)\big\> - \|C_1^{-1}z\|^2 - \|C_2^{-1}w\|^2\\
&= \big\<z + M^{-2}(C_2^{-2}w-C_2^{-2}z),C_1^{-2}z+C_2^{-2}w\big\> - \|C_1^{-1}z\|^2 - \|C_2^{-1}w\|^2\\
&= \<z,C_2^{-2}w\> + \big\<C_2^{-2}(w-z),(C_1^{-2}+C_2^{-2})^{-1}(C_1^{-2}z+C_2^{-2}w)\big\> - \|C_2^{-1}w\|^2\\
&= \<z,C_2^{-2}w\> + \big\<C_2^{-2}(w-z),w + M^{-2}(C_1^{-2}z-C_1^{-2}w)\big\> - \|C_2^{-1}w\|^2\\
&= \big\<C_2^{-2}(w-z),M^{-2}C_1^{-2}(z-w)\big\> = \big\<(C_2^{2}M^{2}C_1^{2})^{-1}(w-z),(z-w)\big\>\\
&= -\<(C_1^2+C_2^2)^{-1}(z-w),(z-w)\big\> = -(z-w)^\top C^{-2}(z-w).
\end{align*}
Thus, we obtain
\begin{align*}
\int_{\R^d}k_z^{C_1}(y)\cdot k_w^{C_2}(y)\,dy
= e^{K(z,w)}\int_{\R^d}\exp(-\|My - M^{-1}(C_1^{-2}z + C_2^{-2}w)\|^2)\,dy = k^C(z,w)\cdot\frac{\pi^{d/2}}{\det M},
\end{align*}
and the lemma is proved.
\end{proof}

\section{Proof of Theorem 1.1}\label{s:proof}
We first prove the following proposition on the action of the Koopman operator of \eqref{e:sde} on $\bH_C$.

\begin{prop}\label{p:immer}
Let $t\ge 0$, let $C\in\R^{d\times d}$ be symmetric positive definite, and define
\begin{align}\label{e:diss}
C_t = \big[e^{-At}(C^2 + 2\Sigma(t))e^{-A^\top t}\big]^{1/2}
\qquad\text{and}\qquad
\tau_t = \frac{\det C}{(\det(C^2 + 2\Sigma(t)))^{1/2}}.
\end{align}
Then the Koopman operator $K^t$, associated with the SDE \eqref{e:sde} and defined in \eqref{e:koopman}, maps $\bH_C$ boundedly into $\bH_{C_t}$ with norm not exceeding $\tau_t^{1/2}$.
\end{prop}
\begin{proof}
Let $z\in\R^d$. By Lemma \ref{l:rechnen}, we have
\begin{align*}
(K^tk_z^C)(x)
&= (2\pi)^{-d/2}(\det\Sigma(t))^{-1/2}\int_{\R^d}k_z^C(y)\exp\big(-\tfrac 12(y-e^{At}x)\Sigma(t)^{-1}(y-e^{At}x)\big)\,dy\\
&= (2\pi)^{-d/2}(\det\Sigma(t))^{-1/2}\int_{\R^d}k_z^C(y)k_{e^{At}x}^{(2\Sigma(t))^{1/2}}(y)\,dy\\
&= (2\pi)^{-d/2}(\det\Sigma(t))^{-1/2}\frac{\pi^{d/2}}{\det(C^{-2} + (2\Sigma(t))^{-1})^{1/2}}k^{(C^2 + 2\Sigma(t))^{1/2}}(z,e^{At}x)\\
&= \frac{(\det\Sigma(t))^{-1/2}}{\det(2C^{-2} + \Sigma(t)^{-1})^{1/2}}\cdot k^{C_t}_{e^{-At}z}(x) = 
\tau_t\cdot k^{C_t}_{e^{-At}z}(x),
\end{align*}
hence, $K^tk_z^C = \tau_t\cdot k^{C_t}_{e^{-At}z}\in\bH_{C_t}$. Therefore, if $f = \sum_{j=1}^n\alpha_jk^C_{x_j}$ with $\alpha_j\in\R$ and $x_j\in\R^d$ ($j=1,\ldots,n$), we obtain
\begin{align*}
\|K^tf\|_{C_t}^2
&= \Bigg\|\sum_{j=1}^n\alpha_jK^tk^C_{x_j}\Bigg\|_{C_t}^2 = \tau_t^2\Bigg\|\sum_{j=1}^n\alpha_jk^{C_t}_{e^{-At}x_j}\Bigg\|_{C_t}^2 = \tau_t^2\sum_{i,j=1}^n\alpha_i\alpha_jk^{C_t}(e^{-At}x_i,e^{-At}x_j)\\
&= \tau_t^2\sum_{i,j=1}^n\alpha_i\alpha_jk^{(C^2+2\Sigma(t))^{1/2}}(x_i,x_j)\le \tau_t^2\cdot\frac{\det(C^2+2\Sigma(t))^{1/2}}{\det C}\sum_{i,j=1}^n\alpha_i\alpha_jk^{C}(x_i,x_j) = \tau_t\|f\|_C^2,
\end{align*}
where the inequality is due to Proposition \ref{p:inc_gen}.

This shows that $K^t$ maps $\bH_{0,C} = \linspan\{k^C_x : x\in\R\}\subset\bH_C$ boundedly into $\bH_{C_t}$. Since $\bH_{0,C}$ is dense in $\bH_C$, it follows that $K^t|_{\bH_{0,C}}$ extends to a bounded operator $T : \bH_C\to\bH_{C_t}$. In order to see that $Tf = K^tf$ for $f\in\bH_C$, let $(f_n)\subset\bH_{0,C}$ such that $f_n\to f$ in $\bH_C$. Then $K^tf_n = Tf_n\to Tf$ in $\bH_{C_t}$. Since $\bH_C\cinc L^2(\mu)$, we have $f_n\to f$ in $L^2(\mu)$ and thus $K^tf_n\to K^tf$ in $L^2(\mu)$. Also, $K^tf_n\to Tf$ in $L^2(\mu)$. Hence, $K^tf=Tf$ $\mu$-a.e.\ on $\R^d$. But as both $K^tf$ and $Tf$ are continuous and $\mu$ is absolutely continuous w.r.t.\ Lebesgue measure with a positive density, we conclude that $K^tf = Tf\in\bH_C$.
\end{proof}

We are now in the position to prove the main result of this note---that the RKHS $\bH_C$ is invariant under the Koopman operator of the SDE \eqref{e:sde} if the Lyapunov-like condition \eqref{e:cond} on the interplay of the matrices $A$, $B$, and $C$ is satisfied.

\begin{proof}[Proof of Theorem \ref{t:invariance}]
Let $C_t$ and $\tau_t$ be defined as in Proposition \ref{p:immer}. We prove that $C_t^2\ge C^2$ for all $t\ge 0$ if any only if \eqref{e:cond} is satisfied. For $x\in\R^d$ and $t\ge 0$ we have
\begin{align*}
f_x(t) &:= \<C_t^2x,x\> = \<e^{-At}(C^2 + 2\Sigma(t))e^{-A^\top t}x,x\> = \|Ce^{-A^\top t}x\|^2 + 2\<e^{-At}\Sigma(t)e^{-A^\top t}x,x\>\\
&= \|Ce^{-A^\top t}x\|^2 + 2\Big\<\Big(\int_0^t e^{-As}BB^\top e^{-A^\top s}\,ds\Big)x,x\Big\> = \|Ce^{-A^\top t}x\|^2 + 2\int_0^t\big\|B^\top e^{-A^\top s}x\big\|^2\,ds.
\end{align*}
Moreover,
\begin{align*}
\dot f_x(t)
&= \big\<Ce^{-A^\top t}x,\tfrac d{dt}Ce^{-A^\top t}x\big\> + \big\<\tfrac d{dt}Ce^{-A^\top t}x,Ce^{-A^\top t}x\big\> + 2\big\|B^\top e^{-A^\top t}x\big\|^2\\
&= -\big\<AC^2e^{-A^\top t}x,e^{-A^\top t}x\big\> - \big\<C^2A^\top e^{-A^\top t}x,e^{-A^\top t}x\big\>  + 2\big\|B^\top e^{-A^\top t}x\big\|^2\\
&= \big\<\big[2BB^\top - (AC^2+C^2A^\top)\big]e^{-A^\top t}x,e^{-A^\top t}x\big\>.
\end{align*}
Note that $f_x(0) = \|Cx\|^2$. Hence, if \eqref{e:cond} holds, we have $\dot f_x(t)\ge 0$ for all $t\ge 0$ and all $x\in\R^d$, which yields $C_t^2\ge C^2$ for all $t\ge 0$. Conversely, if $C_t^2\ge C^2$ for all $t\ge 0$, i.e., $f_x(t)\ge f_x(0)$ for all $x\in\R^d$ and all $t\ge 0$, then $\dot f_x(0)\ge 0$ for all $x\in\R^d$, which is \eqref{e:cond}.

We may thus apply Proposition \ref{p:inc_gen} to conclude that $\bH_{C_t}\cinc\bH_C$ with $\|f\|_C^2\le (\frac{\det C_t}{\det C})\|f\|_{C_t}^2$ for $f\in\bH_{C_t}$ and all $t\ge 0$. Thus,
\begin{align*}
\|K^tf\|_C^2
&\le\frac{\det C_t}{\det C}\|K^tf\|_{C_t}^2 \le \tau_t\frac{\det C_t}{\det C}\|f\|_C^2 = \det e^{-At}\|f\|_C^2 = e^{t\cdot\Tr(-A)}\|f\|_C^2,
\end{align*}
which concludes the proof of the theorem.
\end{proof}

\begin{rem}
It is not clear whether the statement of Theorem \ref{t:invariance} is actually an equivalence. For the other direction one would have to show that $K^t\bH_C\subset\bH_C$ implies $\bH_{C_t}\cinc\bH_C$ as the latter was shown to be equivalent to \eqref{e:cond} in the proof of Theorem \ref{t:invariance}. Note that $K^t\bH_C\subset\bH_C$ at least implies $\bH_{0,C_t} = \linspan\{k_x^{C_t} : x\in\R^d\}\subset\bH_C$, see Proposition \ref{p:immer}.
\end{rem}

\section*{Acknowledgments}
F. Philipp was funded by the Carl Zeiss Foundation within the project {\it DeepTurb--Deep Learning in and from Turbulence}. He was further supported by the free state of Thuringia and the German Federal Ministry of Education and Research (BMBF) within the project {\it THInKI--Th\"uringer Hochschulinitiative für KI im Studium}.
K. Worthmann gratefully acknowledges funding by the Deutsche Forschungsgemeinschaft (DFG, German Research Foundation) -- Project-ID 507037103.

\appendix
\section{Some background on RKHS theory}

\smallskip
\begin{thm}[Aronszajn's inclusion theorem, {\cite[Theorem 5.1]{pr}}]\label{t:aronszajn}
Let $k_1$ and $k_2$ be two symmetric positive definite kernels on a set $X$. Denote their corresponding RKHS's by $\bH_1$ and $\bH_2$, respectively. Then $\bH_1\subset\bH_2$ if and only if $\bH_1\cinc\bH_2$. Moreover, for $c>0$ we have $\bH_1\lhook\joinrel\xrightarrow{\hspace{.1cm}c\hspace{.1cm}}\bH_2$ with $\|f\|_2\le c\|f\|_1$ for $f\in\bH_1$ if and only if $k_1\preceq  c^2 k_2$.
\end{thm}



Given two symmetric positive definite kernels $k_X$ on $X$ and $k_Y$ on $Y$, we define the {\em tensor product} (cf.\ \cite[Definition 5.12]{pr}) $k_X\otimes k_Y : (X\times Y)^2\to\R$ of $k_X$ and $k_Y$ by
\begin{align}\label{e:tensor}
(k_X\otimes k_Y)\big((x,y),(x',y')\big) = k_X(x,x')\cdot k_Y(y,y').
\end{align}
It follows from the positive semi-definiteness of the Schur product of two positive semi-definite matrices (see, e.g., \cite[Theorem 4.8]{pr}) that $k_X\otimes k_Y$ is a symmetric positive definite kernel on $X\times Y$.

\begin{lem}\label{l:tensor}
Let $k_X,k_X'$ be kernels on $X$ and $k_Y,k_Y'$ kernels on $Y$ and assume that $k_X\preceq k_X'$ and $k_Y\preceq k_Y'$. Then also $k_X\otimes k_Y\preceq k_X'\otimes k_Y'$.
\end{lem}
\begin{proof}
Denote the RKHS corresponding to $k_X$, $k_X'$, $k_Y$, and $k_Y'$ by $\bH_X$, $\bH_X'$, $\bH_Y$, and $\bH_Y'$, respectively. Let $(e_s)_{s\in S}$ and $(f_t)_{t\in T}$ be orthonormal bases of $\bH_Y$ and $\bH_X'$, respectively. Then we have ($x_1,x_2\in X$, $y_1,y_2\in Y$)
\[
k_Y(y_1,y_2) = \sum_{s\in S}e_s(y_1)e_s(y_2)
\qquad\text{and}\qquad
k_X'(x_1,x_2) = \sum_{t\in T}f_t(x_1)f_t(x_2)
\]
with pointwise convergence, see \cite[Theorem 2.4]{pr}. Let $n\in\N$ and $\alpha_j\in\R$, $(x_j,y_j)\in X\times Y$, $j=1,\ldots,n$. Then
\begin{align*}
\sum_{i,j=1}^n\alpha_i\alpha_j(k_X\otimes k_Y)\big((x_i,y_i),(x_j,y_j)\big)
&= \sum_{i,j=1}^n\alpha_i\alpha_j k_X(x_i,x_j)k_Y(y_i,y_j)\\
&= \sum_{s\in S}\sum_{i,j=1}^n[\alpha_ie_s(y_i)][\alpha_je_s(y_j)]k_X(x_i,x_j)\\
&\le \sum_{s\in S}\sum_{i,j=1}^n[\alpha_ie_s(y_i)][\alpha_j e_s(y_j)]k_X'(x_i,x_j)\\
&= \sum_{i,j=1}^n\alpha_i\alpha_j k_X'(x_i,x_j)k_Y(y_i,y_j)\\
&= \sum_{t\in T}\sum_{i,j=1}^n[\alpha_i\,f_t(x_i)][\alpha_j\,f_t(x_j)]k_Y(y_i,y_j)\\
&\le \sum_{i,j=1}^n\alpha_i\alpha_j k_X'(x_i,x_j)k_Y'(x_i,y_j)\\
&= \sum_{i,j=1}^n\alpha_i\alpha_j(k_X'\otimes k_Y')\big((x_i,y_i),(x_j,y_j)\big),
\end{align*}
which was to be proven.
\end{proof}

\section*{Author affiliations}
\end{document}